\documentclass[a4paper,11pt]{amsart}
\title{Exponential rarefaction of  maximal real algebraic hypersurfaces}
\author{Michele Ancona}
\thanks{Tel Aviv University, School of Mathematical Sciences; e-mail: \url{michi.ancona@gmail.com}. The author is supported by the Israeli Science Foundation through the ISF Grants 382/15 and 501/18.}
\date{}
\usepackage[a4paper,bindingoffset=0.2in,%
            left=1in,right=1in,top=1in,bottom=1in,%
            footskip=.25in]{geometry}
\usepackage[latin1]{inputenc}
\usepackage[T1]{fontenc}
\usepackage[english]{babel}
\usepackage{amsmath,amssymb,amsthm}
\usepackage{amsfonts}%
\usepackage{amssymb}%
\usepackage{indentfirst}
\usepackage{fancyhdr}
\usepackage{physics}
\usepackage{young}
\usepackage{graphicx}

\usepackage[vcentermath]{youngtab}

\RequirePackage[colorlinks,linkcolor=blue,citecolor=blue,urlcolor=blue]{hyperref}
\usepackage{enumerate}
\usepackage{color}
\usepackage{pst-fill,pst-grad,pst-plot,pst-eucl,pstricks-add,pst-node}
\usepackage{mathdots}
\usepackage{dsfont}

\theoremstyle{plain}
\newtheorem{thm}{Theorem}[section]
\newtheorem{lemma}[thm]{Lemma}
\newtheorem{prop}[thm]{Proposition}

\newtheorem{oss}[thm]{Remark}
\newtheorem{example}[thm]{Example}

\theoremstyle{definition}
\newtheorem{defn}[thm]{Definition}
\newtheorem{conv}[thm]{Notation}

\newcommand{\R}{\mathbb{R}}
\newcommand{\C}{\mathbb{C}}

\newcommand{\Vol}{\textrm{Vol}}

\newcommand{\conj}{\textit{conj}}

\allowdisplaybreaks

\begin{document}
\maketitle
\begin{abstract} 
Given an ample real Hermitian holomorphic line bundle $L$ over a real algebraic variety $X$, the space of real holomorphic   sections of $L^{\otimes d}$ inherits a natural Gaussian probability  measure.  We prove that the probability that the zero locus of a real holomorphic section $s$ of $L^{\otimes d}$ defines a maximal hypersurface tends to $0$ exponentially fast as $d$ goes to infinity. This extends to any dimension a result of  Gayet and  Welschinger \cite{gwexp} valid for maximal real algebraic curves inside a real algebraic surface.

The starting point is a low degree approximation property which relates the topology of the real  vanishing locus of a real holomorphic section of $L^{\otimes d}$ with the topology of the real  vanishing locus  a real holomorphic section of $L^{\otimes d'}$ for a sufficiently smaller $d'<d$. Such a statement is inspired by the recent work of Diatta and Lerario \cite{diatta}.
\end{abstract}
\section{Introduction}
\subsection{Real algebraic varieties and maximal hypersurfaces}
Let $(X,c_X)$ be a real algebraic variety, that is a complex (smooth, projective and connected) algebraic variety equipped with an antiholomorphic involution $c_X:X\rightarrow X$, called the real structure. For example, the projective space $\C P^n$ equipped with the standard conjugaison $\textit{conj}$ is a real algebraic variety. More generally, the solutions of a system of homogeneous real polynomial equations in $n+1$ variables define a real algebraic variety $X$ inside $\C P^n$, whose real structure  is the restriction of $\textit{conj}$ to $X$.
The real locus $\R X$ of a real algebraic variety is the set  of fixed points of the real structure, that is $\R X=\mathrm{Fix}(c_X)$. The real locus $\R X$ is either empty, or a finite union of $n$ dimensional $\mathcal{C}^{\infty}$-manifolds, where $n$ is the (complex) dimension of $X$. The study of the topology of real algebraic varieties is a central topic in real algebraic geometry since the works of Harnack and Klein on the topology of real algebraic curves \cite{harnack,klein} and the famous Hilbert's sixteenth problem \cite{wilson}.
A fundamental restriction on the topology of the real locus of a real algebraic variety is given by the Smith-Thom inequality \cite{thomreelle} which asserts that the total Betti number of the real locus $\R X$ of a real algebraic variety is bounded from above by the total Betti number of the complex locus:  
\begin{equation}\label{smith-thom}
\sum_{i=0}^n\dim H_i(\R X, \mathbb{Z}/2)\leq \sum_{i=0}^{2n}\dim H_i(X, \mathbb{Z}/2).
\end{equation}
We will more compactly write $b_*(\R X)\leq b_*(X)$, where $b_*$ denotes the total Betti number (i.e. the sum of the $\mathbb{Z}/2$-Betti numbers). 
For a real algebraic curve $C$, the inequality \eqref{smith-thom} is known as Harnack's inequality \cite{harnack,klein} and reads $b_0(\R C)\leq g+1$, where $g$ is the  genus of $C$. 
A real algebraic variety which realizes the equality in \eqref{smith-thom} is called  \emph{maximal}. 

In this paper, we will study  maximal real algebraic hypersurfaces inside a fixed real algebraic variety $X$. More precisely, we are interested in the following question: given a real linear system of divisors in $X$, what is the probability to find a divisor defining a  real algebraic maximal hypersurface? The goal of the paper is to show that real algebraic maximal hypersurfaces are exponentially rare inside their linear system.
\subsection{Real Hermitian line bundles and Gaussian measures}\label{secgaussian} In order to answer the previous question, let  $\pi:(L,c_L)\rightarrow (X,c_X)$ be an ample real holomorphic line bundle over $X$, that is, an ample holomorphic line bundle $L$ over $X$ equipped with a real structure $c_L$ such that $\pi\circ c_L=c_X\circ \pi$. We equip $L$ with a smooth real Hermitian metric $h$  of positive curvature $\omega$ (we recall that \textit{real} means $c_L^*h=\bar{h}$). We denote by $\R H^0(X,L^d)$ the space of real global sections of $L^{\otimes d}=L^d$, that is the space of holomorphic sections $s\in H^0(X,L^d)$ such that $s\circ c_X=c_{L^d}\circ s$. This space is naturally equipped with a   $\mathcal{L}^2$-scalar product defined by
\begin{equation}\label{l2 scalar}
\langle s_1,s_2\rangle_{\mathcal{L}^2}=\int_X h^d(s_1,s_2)\frac{\omega^{\wedge n}}{n!}
\end{equation}
for any pair of real global sections $s_1,s_2\in\R H^0(X,L^d)$, where $h^d$ is the real Hermitian metric on $L^d$ induced by $h$.
In turn, the $\mathcal{L}^2$-scalar product \eqref{l2 scalar} naturally induces   a Gaussian probability measure $\mu_d$ defined by
\begin{equation}\label{gaussian measure}
\mu_d(A)=\frac{1}{\sqrt{\pi}^{N_d}}\int_{s\in A}e^{-\norm{s}^2_{\mathcal{L}^2}}\mathrm{d}s
\end{equation}
for any open set $A\subset \R H^0(X,L^d)$, where $N_d$ is the dimension of $\R H^0(X,L^d)$ and $\mathrm{d}s$ the Lebesgue measure induced by the $\mathcal{L}^2$-scalar product \eqref{l2 scalar}. 
The probability space we will consider is then $\big(\R H^0(X,L^d),\mu_d\big)$.
\begin{example}\label{Kost} When $(X,c_X)$ is the $n$-dimensional projective space and $(L,c_L,h)$ is the degree $1$ real holomorphic line bundle  equipped with the standard Fubini-Study metric, then the vector space $\R H^0(X,L^{d})$ is isomorphic to the space $\R_d^{hom}[X_0,\dots,X_n]$ of degree $d$ homogeneous real polynomials in $n+1$ variables and the $\mathcal{L}^2$-scalar product is the one which makes the family of monomials $\big\{\sqrt{\binom{(n+d)!}{n!\alpha_0!\cdots\alpha_n!}}X_0^{\alpha_0}\cdots X_n^{\alpha_n}\big\}_{\alpha_0+\dots+\alpha_n=d}$ an orthonormal basis. A random polynomial with respect to the Gaussian probability measure induced by this scalar product is called a Kostlan polynomial.
\end{example}
\subsection{Statement of the main results}
The zero locus of a real global section $s_d$ of $L^d$ is denoted by $Z_{s_d}$ and its real locus by $\R Z_{s_d}$. The total Betti number $b_*(Z_{s_d})$ of $Z_{s_d}$ verifies the asymptotic $b_*(Z_{s_d})= v(L)d^n+O(d^{n-1}),$ as $d\rightarrow \infty$, where $v(L):=\int_{X}c_1(L)^n$ is the so-called volume of the line bundle $L$ (see, for example, \cite[Lemma 3]{gw1}).  
If the total Betti number $b_*(\R Z_{s_d})$ of the real locus of a sequence of real algebraic hypersurfaces verifies  the same asymptotics, then the hypersurfaces are called \emph{asymptotically maximal}. The existence of asymptotically maximal hypersurfaces is known in many cases, for example for real algebraic surfaces \cite[Theorem 5]{gwexp}, for  projective spaces \cite{itenbergviro} and for toric varieties \cite[Theorem 1.3]{bertrand}. 
 The first main result of the paper shows that asymptotically maximal hypersurfaces are very rare in their linear system. More precisely:
\begin{thm}\label{theorem rarefaction} Let $(X,c_X)$ be a real algebraic variety of dimension $n$ and $(L,c_L)$ be a  real  Hermitian line bundle  of positive curvature.
Then, there exists a positive $a_0<v(L)$, such that, for any $a> a_0$, we have 
$$\mu_d\big\{s\in\R H^0(X,L^d), b_*(\R Z_s)\geq  ad^n\big\}\leq O(d^{-\infty})$$
as $d\rightarrow\infty$.
\end{thm}
The notation $O(d^{-\infty})$ stands for $O(d^{-k})$ for any $k\in\mathbb{N}$ and the measure $\mu_d$ is the Gaussian measure defined in Equation \eqref{gaussian measure}. 
Note that actually we have more than "asymptotically maximal hypersurfaces are very rare". Indeed, asymptotically maximal hypersurfaces corresponds to the asymptotics $b_*(\R Z_s)=v(L)d^n+O(d^{n-1})$, while in Theorem \ref{theorem rarefaction} we consider bigger subsets of $\R H^0(X,L^d)$  of the form   $b_*(\R Z_s)\geq ad^n$, for   $v(L)>a> a_0$.

For maximal real algebraic hypersurfaces, the rarefaction is even exponential. 
More precisely:
\begin{thm}\label{theorem exponential rarefaction} Let $(X,c_X)$ be a real algebraic variety of dimension $n$ and $(L,c_L)$ be a  real  Hermitian line bundle  of positive curvature.
For any $a>0$  there exists $c>0$ such that 
$$\mu_d\big\{s\in\R H^0(X,L^d), b_*(\R Z_s)\geq b_*( Z_s)-ad^{n-1}\big\}\leq O(e^{-c\sqrt{d}\log d})$$
as $d\rightarrow\infty$.
Moreover, if the real Hermitian metric on $L$ is  analytic,  one has the estimate
$$\mu_d\big\{s\in\R H^0(X,L^d), b_*(\R Z_s)\geq b_*( Z_s)-ad^{n-1}\big\}\leq O(e^{-cd}).$$
\end{thm}
Theorems \ref{theorem exponential rarefaction} extends to any dimension a result of  Gayet and Welschinger \cite{gwexp}, in which the authors prove, using the theory of laminary currents, that    maximal real curves are exponentially rare in a real algebraic surface.
We stress that our techniques are different from those of \cite{gwexp}.
\begin{oss}\label{fubinistudy} The $\mathcal{L}^2$-scalar product on $\R H^0(X,L^d)$ also induces a Fubini-Study volume on the linear system $P(\R H^0(X,L^d))$. Both sets considered in Theorems \ref{theorem rarefaction} and \ref{theorem exponential rarefaction} are cones in $\R H^0(X,L^d)$ and the volume (with respect to the Fubini-Study form) of their projectivization   coincides with the Gaussian measures estimated in Theorems \ref{theorem rarefaction} and \ref{theorem exponential rarefaction}.
\end{oss}
Theorems \ref{theorem rarefaction} and \ref{theorem exponential rarefaction}  are consequence of a  more general "approximation theorem" which states that, for some $b<1$, with very high probability, the zero locus of a real section of $L^d$ is diffeomorphic to the zero locus of a real section of $L^{\lfloor bd \rfloor}$, where $\lfloor bd \rfloor$ is the greatest integer less than or equal to $bd$. More precisely we have:
\begin{thm}\label{theorem approximation} Let $(X,c_X)$ be a real algebraic variety  and $(L,c_L)$ be a  real Hermitian line bundle  of positive curvature. 
\begin{enumerate}
\item There exists a positive $b_0<1$ such that for any $b_0< b <1$ the following happens:  the probability that,  for a real  section $s$ of $L^d$, there exists a real section $s'$ of $L^{\lfloor bd \rfloor}$ such that the pairs $(\R X,\R Z_s)$ and $(\R X,\R Z_{s'})$ are isotopic,  is at least $1-O(d^{-\infty})$, as $d\rightarrow\infty$.
\item  For any $k\in \mathbb{N}$ there exists $c>0$ such that the following happens: the probability that, for a real section $s$ of $L^d$, there exists a real section $s'$ of $L^{d-k}$ such that the pairs $(\R X,\R Z_s)$ and $(\R X,\R Z_{s'})$ are isotopic,  is at least $1-O(e^{-c\sqrt{d}\log d})$, as $d\rightarrow\infty$. If moreover the real Hermitian metric on $L$ is  analytic, this probability is  at least $1-O(e^{-cd})$, as $d\rightarrow\infty$
\end{enumerate}
\end{thm}
Assertion \textit{(1)} (resp. assertion \textit{(2)}) of Theorem \ref{theorem approximation}, together with the Smith-Thom inequality \eqref{smith-thom} and of the asymptotics $b_*(Z_{s_d})= v(L)d^n+O(d^{n-1}),$ for $s_d\in \R H^0(X,L^d)$, will imply Theorem \ref{theorem rarefaction} (resp. Theorem \ref{theorem exponential rarefaction}).  
It is also worth pointing out that  the assertions \textit{(1)} and \textit{(2)} of Theorem \ref{theorem approximation} are independent of each other, although their proofs, which will be sketched in Section \ref{SecProof}, are similar.\
 \begin{oss}  Theorem \ref{theorem approximation} implies not only that maximal hypersurfaces are rare, but that "maximal configurations" are.  For instance, we will show in Section \ref{secnest} that in some suitable real algebraic surfaces the probability that a real algebraic curve has a deep nest of ovals is exponentially small (roughly speaking a nest of ovals means severals ovals one inside the other), see Theorems \ref{hirzethm} and \ref{delpezzothm}.
 \end{oss}
 
When $(X,c_X)=(\mathbb{C}P^n,\conj)$, $L=\mathcal{O}(1)$ and the Hermitian metric on $L$ equals the Fubini-Study metric (that is for the case of the so-called Kostlan polynomials, see Example \ref{Kost}), a low degree approximation property was recently proved by   Diatta and  Lerario \cite{diatta} (see also \cite{breiding}), so that Theorem \ref{theorem approximation} is a natural generalization of their result. Actually, in \cite{diatta} the authors proves a general low degree approximation property for Kostlan polynomials: for instance, they have proved that a degree $d$ Kostlan polynomial can be approximated (in the sense of Theorem \ref{theorem approximation}) by a degree $b\sqrt{d\log d}$ Kostlan polynomial, $b>0$, with probability $1-O(d^{-a})$, with $a>0$ depending on  $b$.  

We stress that in \cite{diatta}, it is essential that the considered real algebraic variety is the projective space and that the  metric on $\mathcal{O}(1)$ is the Fubini-Study one (and not, for instance, a small perturbation of it). Indeed, in this situation: 
\begin{enumerate}
\item the induced $\mathcal{L}^2$-scalar product on $\R_d^{hom}[X_0,\dots,X_n]$ is invariant  under the action of the orthogonal group $O(n+1)$, which acts on the variables $X_0,\dots,X_n$ (equivalently, the group $O(n+1)$ acts by real holomorphic isometries on $(\C P^n,\conj)$);
\item there exists a canonical $O(n+1)$-invariant decomposition  $\R_d^{hom}[X_0,\dots,X_n]=\bigoplus_{d-\ell\in 2 \mathbb{N}}V_{d,\ell}$, where $V_{d,\ell}$ is the space of  homogeneous harmonic polynomials  of degree $\ell$, thanks to which it is possible to define projections of degree $d$ polynomials into  lower degree ones. 
\end{enumerate}
 These two properties, together with the classification \cite{kostscalarproduct} of  the $O(n+1)$-invariant scalar products on $\R_d^{hom}[X_0,\dots,X_n]$, are fundamental for the proof of the results in \cite{diatta}. This  is a very special feature of Kostlan polynomials  and the reason why in our general case some of the approximations of \cite{diatta} cannot be obtained.  Indeed,
  \begin{enumerate}
\item on a general real algebraic variety equipped with a K\"ahler metric $\omega$ the group of holomorphic isometries is trivial;
\item given a real Hermitian holomorphic line bundle $L\rightarrow X$ there is no canonical decomposition of $\R H^0(X,L^d)$.
\end{enumerate}
Hence, in order to obtain an approximation property for sections of line bundles on a general real algebraic variety,  we have to use a different strategy, that we will explain in more details in  Section \ref{Comments on proof}. In particular, in our proof, in constrast to the case of polynomials \cite{diatta}, the complex locus of the variety $X$ (and not only the real one) plays a fundamental role. Indeed, we will consider real subvarieties of $X$ with empty real loci and we will study the real sections of $L^d$ that vanish along these subvarieties. These real sections are the fundamental tool for our low degree approximation property.

 During the proof of Theorem \ref{theorem approximation}, we will also need to understand how much we can perturb a real section $s\in\R H^0(X,L^d)$ without changing the topology of its real locus. This leads us to  study two quantities  related to the discriminant $\R\Delta_d\subset\R H^0(X,L^d)$, that is  the subset of sections which do not vanish transversally along $\R X$. 
 More precisely, we will consider the volume (with respect to the Gaussian measure $\mu_d$) of  tubular  neighborhoods of $\R\Delta_d$ and the function "distance to the discriminant". 
Such quantities have already been observed to be important in the case of Kostlan polynomials in \cite[Section 4]{diatta}, however in our general framework the lack of symmetries makes the computation of those  more delicate and requires the use of the Bergman kernel's estimates along the diagonal \cite{ber1,zel,ma2}.

In particular, denoting by $\mathrm{dist}_{\R\Delta_d}(s)$ the distance (induced by the $\mathcal{L}^2$-scalar product \eqref{l2 scalar}) from a section $s$ to the discriminant $\R\Delta_d$, we obtain the uniform estimate (see Lemma \ref{distance to the discriminant}):
 $$\mathrm{dist}_{\R\Delta_d}(s)=\min_{x\in\R X}\bigg(\frac{\norm{s(x)}^2_{h^d}}{d^n}+\frac{\norm{\nabla s(x)}^2_{h^d}}{d^{n+1}}\bigg)^{1/2}\big(\pi^{n/2}+O(1/d)\big)
$$
as $d\rightarrow\infty$. We believe  that this formula has an independent interest.\\

Finally, let us conclude this section by pointing out that, while in the present paper we are interested in some rare events  (that is, real algebraic hypersurfaces with rich topology),  the expected value $\mathbb{E}[b_i(\R Z_{s_d})]$ of the $i$-th  Betti number of $\R Z_{s_d}$ is bounded from below and from above by respectively $c_{i,n}d^{n/2}$ and $C_{i,n}d^{n/2}$, for some positive constants $0<c_{i,n}\leq C_{i,n}$, see \cite{gw2,gwlowerbound}. Moreover, in the case of the $0$-th Betti number (that is, for the number of connected components), it is known that $\lim_{d\rightarrow\infty}\frac{1}{d^{n/2}}\mathbb{E}[b_0(\R Z_{s_d})]$ exists and is positive \cite{nazarovsodinasymptotic}. 
\subsection{Idea of the proof of Theorem \ref{theorem approximation}}\label{Comments on proof}
In this section, we  sketch the proof of Assertion \textit{(2)}  of Theorem \ref{theorem approximation}, the proof of Theorem \ref{theorem approximation} \textit{(1)} is  similar.  

We want to prove that, with very high probability, the real vanishing locus of a real section $s$ of $L^d$ is ambient  isotopic to the real vanishing locus of a real section $s'$ of $L^{d-k}$. 

 The first fact  we will use  is the existence of a real  section $\sigma$ of $L^{k}$, for some suitable even $k\in 2\mathbb{N}$ large enough, with the properties that $\sigma$ vanishes transversally and $\R Z_{\sigma}=\emptyset$ (see Proposition \ref{existance of sigma}). In order to obtain such section $\sigma$, we consider an integer $m$  such that $L^m$ is very ample and a basis $\{s_1,\dots,s_N\}$ is a basis of $\R H^0(X,L^m)$. Then, $\sigma$ is any general small perturbation  of the  section $\sum_{i=1}^{N}s_i^{\otimes 2}$ of $L^{2m}$. 

Let us define $\R H_{d,\sigma}$ to be the vector space of real global sections $s\in\R H^0(X,L^d)$ such that $s$ vanishes along the vanishing locus $Z_\sigma$ of $\sigma$.
We also denote by $\R H^{\perp}_{d,\sigma}$ the orthogonal complement of $\R H_{d,\sigma}$  with respect to the $\mathcal{L}^2$-scalar product defined in Equation \eqref{l2 scalar}.
 Then, for any section $s\in\R H^0(X,L^d)$, there exists an unique decomposition $s=s_{\sigma}^{\perp}+s_{\sigma}^{0}$ with $s_{\sigma}^{0}\in\R H_{d,\sigma}$ and $s_{\sigma}^{\perp}\in\R H^{\perp}_{d,\sigma}$. 
 
 The fundamental point, that we will prove in Section \ref{secbergmankernels} using the theory of logarithmic and partial Bergman kernels \cite{partialbergmankernel,logbergman}, is that  the "orthogonal component" $s_{\sigma}^{\perp}$ of $s$ has a very small $\mathcal{C}^1$-norm along the real locus $\R X$.
A  geometric reason why this is true is that the space $\R H^{\perp}_{d,\sigma}$ is generated by the peak sections \cite{tian} which have a peak on $Z_\sigma$ and,  as $Z_\sigma\cap \R X=\emptyset$,  the pointwise $\mathcal{C}^1$-norm of these peak sections is very small along $\R X$ (indeed, a peak section has a very small $\mathcal{C}^1$-norm on any compact disjoint from its peak).

From the fact that $s_{\sigma}^{\perp}$ has a very small $\mathcal{C}^1$-norm along the real locus $\R X$, we deduce that $s$ is a "small pertubation" of $s_{\sigma}^0$: Thom's Isotopy Lemma would therefore imply that the pairs $(\R X,\R Z_{s})$ and $(\R X,\R Z_{s_{\sigma}^{0}})$ are isotopic \textit{if} $s_{\sigma}^0$ has a big enough $\mathcal{C}^1$-norm along $\R X$. This last implication can be translated in terms of distance from $s_{\sigma}^0$ to the discriminant $\R\Delta_d$, that is the space of real sections which do not vanish transversally along $\R X$: if $s_{\sigma}^0$ is far enough away from the discriminant, then the  pairs $(\R X,\R Z_{s})$ and $(\R X,\R Z_{s_{\sigma}^{0}})$ are isotopic. 

Using the Bergman kernel, we are able  to estimate the function "distance to the discriminant" (Lemma \ref{distance to the discriminant}). These estimates, together with an approach similar to \cite[Section 4]{diatta}, allows us to prove that, with very high probability, $s_{\sigma}^0$ is far enough away from the discriminant, so that,  with very high probability, the pairs $(\R X,\R Z_{s})$ and $(\R X,\R Z_{s_{\sigma}^{0}})$ are isotopic (Lemmas \ref{isotopy} and \ref{rapiddecay}).

% Bergman kernel estimates allow us to estimate the function "distance to the discriminant" (Lemma \ref{distance to the discriminant})  and  to prove that,  with very high probability, $s_{\sigma}^0$ is far enough away from the discriminant, so that,  with very high probability, the pairs $(\R X,\R Z_{s})$ and $(\R X,\R Z_{s_{\sigma}^{0}})$ are isotopic (Lemmas \ref{isotopy} and \ref{rapiddecay}).
   
 Finally, we prove that the section  $s_{\sigma}^{0}$ can be written as $\sigma\otimes s'$, for some real section $s'$ of $L^{d-k}$. In particular, as $\sigma$ does not have any real zero, we have the equality $\R Z_{s_{\sigma}^{0}}=\R Z_{s'}$, which proves Theorem \ref{theorem approximation}.
\subsection{Organization of the paper}
The paper is organized as follows. 

 In Section \ref{SecLowdegree}, we prove the existence of a real section $\sigma$ of $L^k$ with empty real vanishing locus and then we will study the real sections of $L^d$ which vanish along $Z_\sigma$. This leads us to consider logarithmic and partial Bergman kernels in Section \ref{secbergmankernels}. 
 
  In Section \ref{Secdiscr}, we study the geometry of the discriminant $\R\Delta_d\subset\R H^0(X,L^d)$ and compute several quantities related to it: its degree, the volume (with respect to the Gaussian measure $\mu_d$) of  tubular  neighborhoods and the function "distance to the discriminant".  

 In Section \ref{SecProof}, we prove our main results, namely Theorems \ref{theorem rarefaction}, \ref{theorem exponential rarefaction} and \ref{theorem approximation}.
 
 Finally, in Section \ref{secnest}, we study the depth of the nests of real algebraic curves inside some real algebraic surfaces, namely Hirzebruch surfaces (Theorem \ref{hirzethm}) and Del Pezzo surfaces (Theorem \ref{delpezzothm}).

\section{Sections vanishing along a fixed hypersurface}\label{SecLowdegree}
In this section, we prove the existence of a real global section $\sigma$ of $L^k$ with empty real vanishing locus and we  study the real sections of $L^d$ which vanish along $Z_\sigma$. 
\subsection{A global section with empty real vanishing locus} Let $(L,c_L)$ be an ample real holomorphic line bundle over a real algebraic variety $(X,c_X)$. A \textit{real section} of $L$ is a global holomorphic section $s$ of $L$ such that $s\circ c_X=c_L\circ s$.
\begin{prop}\label{existance of sigma} There exists an even positive integer $k_0$ such that for any even $k\geq k_0$ there exists a real  section $\sigma$ of $L^{k}$ with the following properties: (i)  $\sigma$ vanishes transversally and (ii) $\R Z_{\sigma}$ is empty.
\end{prop}
 \begin{proof}
 Let $m_0$ be the smallest integer such that $L^{m_0}$ is very ample and set $k_0=2m_0$. For any integer $m\geq m_0$, fix a  basis $s_1,\dots,s_{N_{m}}$ of $\R H^0(X,L^m)$ and consider the real section $s=\sum_{i=1}^{N_m}s_i^{\otimes 2},$ which is a real section of $L^{2m}$ whose real vanishing locus   is empty. 
  Remark that  any small perturbation of $s$ inside $\R H^0(X,L^{2m})$ will have empty real vanishing locus and that the discriminant (i.e. the sections which do not vanish transversally) is an algebraic hypersurface of $\R H^0(X,L^{2m})$. We can then find a small perturbation $\sigma$ of $s$ which has the desired properties, namely  $\sigma$ is a real global section of $L^{2m}$  vanishing transversally and whose real  vanishing locus  is empty. 
 \end{proof}

 \begin{defn}\label{definitionvanishing} Let $\sigma\in \R H^0(X,L^k)$ be a section given by Proposition \ref{existance of sigma}, where $k$ is a fixed large enough integer, and let $Z_\sigma$ be its vanishing locus. For any integer $\ell$, we will use the notation $\sigma^{\ell}$ to designate the section $\sigma^{\otimes \ell}$ of $L^{k\ell}$. For any pair of integers $d$ and $\ell$, let us  define the space $H_{d,\sigma^\ell}$
 to be the subspace of $H^0(X,L^d)$ consisting of sections which vanish along $Z_{\sigma}$ with order at least $\ell$.  Similarly, the vector space $\R H_{d,\sigma^\ell}$ denotes  the subspace of $\R H^0(X,L^d)$ of real sections which vanish along $Z_{\sigma}$ with order at least $\ell$. 
 \end{defn}

 \begin{prop}\label{vanishing order} Let $\sigma\in \R H^0(X,L^k)$ be a section given by Proposition \ref{existance of sigma}, for some fixed $k$ large enough. 
 For any pair of integers $d$ and $\ell$,  the space $H_{d,\sigma^\ell}$ coincides with the space of section $s\in H^0(X,L^d)$ such that $s=\sigma^{\ell}\otimes s'$, for some $s'\in H^0(X,L^{d-k\ell})$. Similarly, the space $\R H_{d,\sigma^\ell}$ coincides with the space of real section $s\in \R H^0(X,L^d)$ such that $s=\sigma^{\ell}\otimes s'$, for some $s'\in \R H^0(X,L^{d-k\ell})$.
\end{prop}
\begin{proof} If a section $s$ of $H^0(X,L^d)$ is of the form $s=\sigma^{\ell}\otimes s'$, for some $s'\in \R H^0(X,L^{d-k\ell})$, then $s\in H_{d,\sigma^\ell}$ (that is, $s$ vanishes  along $Z_{\sigma}$ with order at least $\ell$). Moreover if such section is real then $s\in\R H_{d,\sigma^\ell}$. Let us then prove the other inclusion.
 Let $s\in H^0(X,L^d)$ be such that $s$ vanishes along $Z_\sigma$ with order at least $\ell$. We want to prove that there exists $s'\in H^{0}(X,L^{d-k})$ such that $s=\sigma^{\ell}\otimes s'$. For this, let $\{U_i\}_i$ be a cover of $X$ by open subsets such that the line bundle $L$ is obtained by gluing together the local models $\{U_i\times\C\}_i$ using the maps $f_{ij}:(x,v)\in\big(U_i\cap U_j\big)\times\C\mapsto (x,g_{ij}(x)v) \in (U_i\cap U_j\big)\times\C$, where $g_{ij}:U_i\cap U_j\rightarrow\C^*$ are holomorphic maps. Remark that, for any integer $d$, the line bundle $L^d$ is obtained by gluing together the same local models  $\{U_i\times\C\}_i$ using the maps  $f^d_{ij}:(x,v)\in \big(U_i\cap U_j\big)\times\C\mapsto (x,g_{ij}^d(x)v) \in(U_i\cap U_j\big)\times\C$. Using these trivializations, a global section $s$ of $L^d$ is equivalent to the data of local holomorphic functions $s_i$ on $U_i$ such that $s_i=g^d_{ij}s_j$ on $U_i\cap U_j$. We can then define locally the section $s'$ we are looking for by setting $s_i'=\frac{s_i}{\sigma_i^{\ell}}$. Indeed, these are holomorphic functions because $s_i$
vanishes along the zero set of $\sigma_i$ with order bigger or equal than $\ell$. With this definition, it is  straightforward to check that the family $\{s_i'\}_i$ glues together and define a global section $s'$ of $L^{d-k\ell}$ with the property that $s=\sigma^{\ell}\otimes s'$.  

The proof for the real case follows from the complex case and from that fact that if $s$ and $\sigma$ are real sections, then $s'$ is also real.
\end{proof} 
\subsection{$\mathcal{L}^2$-orthogonal complement to  $\R H_{d,\sigma^\ell}$ and $\mathcal{C}^1$-estimates}\label{secbergmankernels} In this section, we equip $L$ with a real Hermitian metric $h$ with positive curvature  and  consider the induced $\mathcal{L}^2$-scalar product on $\R H^0(X,L^d)$  given by Equation \eqref{l2 scalar}. The main goal  is to study the real sections of $L^{d}$ which are $\mathcal{L}^2$-orthogonal to the space $\R H_{d,\sigma^\ell}$ defined in Definition \ref{definitionvanishing}.
\begin{defn}[$\mathcal{C}^1$-norm] Let $K\subset X$ be a compact set. We define the $\mathcal{C}^1(K)$-norm of a global section $s$ of $L^d$ to be $\norm{s}_{\mathcal{C}^1(K)}=\max_{x\in K}\norm{s(x)}_{h^d}+\max_{x\in K}\norm{\nabla s(x)}_{h^d}$, where $\norm{\cdot}_{h^d}$ is the norm induced by the Hermitian metric $h^d$, $\nabla$ is the Chern connection on $L^d$ induced by $h$ and $\norm{\nabla s(x)}_{h^d}^2=\displaystyle\sum_{i=1}^n\norm{\nabla_{v_i}s(x)}_{h^d}^2$, with $\big\{v_1,\dots, v_n\big\}$ an orthonormal basis of $T_xX$.
\end{defn}
 \begin{oss} Recall that $\R H_{d,\sigma^\ell}= H_{d,\sigma^\ell}\cap \R H^0(X,L^d)$. The orthogonal complement (with respect to the $\mathcal{L}^2$-scalar product given by \eqref{l2 scalar}) of $\R H_{d,\sigma^\ell}$ inside $\R H^0(X,L^d)$ coincides with $H_{d,\sigma^\ell}^{\perp}\cap \R H^0(X,L^d)$. Here, $H_{d,\sigma^\ell}^{\perp}$ is the orthogonal complement of $H_{d,\sigma^\ell}$ with respect to the $\mathcal{L}^2$-Hermitian product defined by 
\begin{equation}\label{l2hermitian}
 \langle s_1,s_2\rangle_{\mathcal{L}^2}=\int_X h^d(s_1,s_2)\frac{\omega^{\wedge n}}{n!}
 \end{equation}
for any pair of  global sections $s_1,s_2\in H^0(X,L^d)$.
 \end{oss}

\begin{prop}\label{partialbergman} Let $\sigma\in \R H^0(X,L^k)$ be a section given by Proposition \ref{existance of sigma}, for some fixed $k$ large enough. There exists a positive real number $t_0$ such that, for any $t\in(0, t_0)$,  we have the uniform estimate $\norm{\tau}_{\mathcal{C}^1(\R X)}=O(d^{-\infty})$ for any real section $\tau\in\R H_{d,\sigma^{\lfloor t d\rfloor}}^{\perp}$ with $\norm{\tau}_{\mathcal{L}^2}=1$, 
as $d\rightarrow\infty$. Here, $\R H_{d,\sigma^{\lfloor t d\rfloor}}$ is as in Definition \ref{definitionvanishing} and  $\lfloor t d \rfloor$ is the greatest integer less than or equal to $t d$.
\end{prop}
\begin{proof}
Let $\tau_1,\dots,\tau_{m_d}$ be an orthonormal basis of $ H^{\perp}_{d,\sigma^{\lfloor t d\rfloor}}$ and $s_1,\dots,s_{N_d-m_d}$ be an orthonormal basis of $H_{d,\sigma^{\lfloor t d\rfloor}}$ (the Hermitian products are the ones induced by Equation \eqref{l2hermitian}).
 We will denote by $P^{\perp}_d(x)=\sum_{i=1}^{m_d}\norm{\tau_i(x)}^2_{h^d}$ and by $P^{0}_d(x)=\sum_{i=1}^{N_d-m_d}\norm{s_i(x)}^2_{h^d}$, so that $P^{\perp}_d+P^{0}_d$ equals the Bergman function $P_d$ (i.e. the value of the Bergman kernel on the diagonal).
The function $P^0_d$ equals the \textit{partial Bergman kernel} of order $\lfloor t d\rfloor$ associated with the subvariety $Z_\sigma$, see \cite{partialbergmankernel}.
 By \cite[Theorem 1.3]{partialbergmankernel}, for any compact subset $K$ of $X$ which is disjoint from $Z_\sigma$ and for any $r\in\mathbb{N}$, there exists a positive real number $t_0(K)$ such that, for any $t< t_0(K)$,  one has $\norm{P_d-P^0_d}_{\mathcal{C}^r(K)}=O(d^{-\infty})$. Now,  we have $P^{\perp}_d=P_d-P^0_d$, so that, by choosing $K=\R X$, which is disjoint from $Z_\sigma$ by construction of $\sigma$, we obtain $\norm{P^\perp_d}_{\mathcal{C}^r(\R X)}=O(d^{-\infty})$, as soon as $t< t_0=t_0(\R X)$. This implies that for any  $t< t_0$  and  any  $\tau\in H_{d,\sigma^{\lfloor t d\rfloor}}^{\perp}$, with $\norm{\tau}_{\mathcal{L}^2}=1$, one has $\norm{\tau}_{\mathcal{C}^1(\R X)}=O(d^{-\infty})$  and, in particular, this happens if $\tau\in\R H_{d,\sigma^{\lfloor t d\rfloor}}^{\perp}$.
\end{proof}

\begin{prop}\label{logberg} Let $\sigma\in \R H^0(X,L^k)$ be a section given by Proposition \ref{existance of sigma}, for some fixed $k$ large enough. There exists $c>0$ (depending on $k$) such that, we have the uniform estimate $\norm{\tau}_{\mathcal{C}^1(\R X)}\leq O(e^{-c\sqrt{d}\log d})$ for any real section $\tau\in \R H_{d,\sigma}^{\perp}$ with $\norm{\tau}_{\mathcal{L}^2}=1$.  If the real Hermitian metric on $L$ is  analytic, then  we have the uniform estimate $\norm{\tau}_{\mathcal{C}^1(\R X)}\leq O(e^{-c d})$ for any real section $\tau\in \R H_{d,\sigma}^{\perp}$ with $\norm{\tau}_{\mathcal{L}^2}=1$.
\end{prop}
\begin{proof}
The proof follows the same idea of Proposition \ref{partialbergman}. 
Let $\tau_1,\dots,\tau_{m_d}$ be an orthonormal basis of $ H^{\perp}_{d,\sigma}$ and $s_1,\dots, s_{N_d-m_d}$ be an orthonormal basis of $H_{d,\sigma}$ (the Hermitian products are the ones induced by Equation \eqref{l2hermitian}).
 We will denote by $P^{\perp}_d(x)=\sum_{i=1}^{m_d}\norm{\tau_i(x)}^2_{h^d}$ and by $P^{0}_d(x)=\sum_{i=1}^{N_d-m_d}\norm{s_i(x)}^2_{h^d}$, so that $P^{\perp}_d+P^{0}_d$ equals the Bergman function $P_d$.
 The function $P^0_d$ equals the \textit{logarithmic Bergman kernel}  associated with the subvariety $Z_\sigma$, see \cite{logbergman}.
 In \cite[Theorem 3.4]{logbergman}, it is proved that for any sequence of compact  sets $K_d$ whose distance from $Z_\sigma$ is bigger than $\frac{\log d}{\sqrt{d}}$, one has $\norm{P_d-P^0_d}_{\mathcal{C}^r(K_d)}=O(d^{-\infty})$. The proof uses the fact that the $\mathcal{C}^r$-norm of a peak section \cite{tian} at a point $z_d$ (depending on $d$) is $O(d^{-\infty})$, as soon as the distance between the peak and the point $z_d$ is bigger than $\frac{\log d}{\sqrt{d}}$. This  is a direct consequence of the fact that the $\mathcal{C}^r$-norm of the Bergman kernel at a pair of points $(x_d,y_d)$ with $\textrm{dist}(x_d,y_d)>\frac{\log d}{\sqrt{d}}$ is $O(d^{-\infty})$. Now, the Bergman kernel decreases exponentially fast outside any fixed neighborhood of the diagonal, that is its $\mathcal{C}^r$-norm at a \textit{fixed} pair of distinct points $(x,y)$  is  $O(e^{-c\sqrt{d}\log d})$, in the general case, and  $O(e^{-cd})$, if the metric on $L$ is  analytic (see, for instance, \cite[Theorem 1.1 and Corollary 1.4]{hezarilu}). This implies that  the $\mathcal{C}^r$-norm of a peak section at a \textit{fixed, not depending on $d$}, point $z$ is $O(e^{-c\sqrt{d}\log d})$, in the general case, and  $O(e^{-cd})$, if the metric on $L$ is  analytic.
 
%  Now,  the $\mathcal{C}^r$-norm of a peak section at a \textit{fixed, not depending on $d$}, point $z$ is $O(e^{-c\sqrt{d}\log d})$, in the general case, and  $O(e^{-cd})$, if the metric on $L$ is  analytic (this is a direct consequence of the fact that the norm of a peak section is controlled by the  norm of the Bergman kernel, and the latter decreases exponentially fast outside a fixed neighborhood of the diagonal, see for example \cite[Theorem 1.1 and Corollary 1.4]{hezarilu}).
  
Using this fact, the same proof as in  \cite[Theorem 3.4]{logbergman} gives us  $\norm{P_d-P^0_d}_{\mathcal{C}^r(K)}\leq O(e^{-c\sqrt{d}\log d})$ (and $\norm{P_d-P^0_d}_{\mathcal{C}^r(K)}\leq O(e^{-cd})$, in the analytic case), where $K$ is any compact disjoint from $Z_\sigma$ and $c$ is a positive constant depending on $K$ and $r\in\mathbb{N}$. 
 
 Now,  we have $P^{\perp}_d=P_d-P^0_d$, so that, by choosing $K=\R X$, which is disjoint from $Z_\sigma$ by construction of $\sigma$, we obtain $\norm{P^\perp_d}_{\mathcal{C}^r(\R X)}\leq O(e^{-c\sqrt{d}\log d})$. This implies that for any  $\tau\in H_{d,\sigma}^{\perp}$, with $\norm{\tau}_{\mathcal{L}^2}=1$, one has $\norm{\tau}_{\mathcal{C}^1(\R X)}\leq O(e^{-c\sqrt{d}\log d})$  and, in particular, this happens if $\tau\in\R H_{d,\sigma}^{\perp}$. Similarly, for the analytic case, we have $\norm{\tau}_{\mathcal{C}^1(\R X)}\leq O(e^{-cd})$  for any  $\tau\in\R H_{d,\sigma}^{\perp}$ with $\norm{\tau}_{\mathcal{L}^2}=1$.
\end{proof}
\section{Geometry of the discriminant: distance, degree and volume}\label{Secdiscr}
Let $\R\Delta_d\subset\R H^0(X,L^d)$ denote the discriminant, that is the subset of sections which do not vanish transversally along $\R X$.   In this section, we will study and compute several quantities related to the discriminant: its degree, the volume (with respect to the Gaussian measure $\mu_d$) of  tubular conical neighborhoods (see Definition \ref{tubular}) and the function "distance to the discriminant".

As in the previous section,  we  equip $L$ with a real Hermitian metric $h$ with positive curvature and  consider the induced $\mathcal{L}^2$-scalar product on $\R H^0(X,L^d)$  given by Equation \eqref{l2 scalar}. Given $s\in\R H^0(X,L^d)$, the next lemma gives us an explicit formula for computing the distance from $s$ to the discriminant $\R\Delta_d$. This distance is computed with respect to the $\mathcal{L}^2$-scalar product, that is   $$\mathrm{dist}_{\R\Delta_d}(s):=\min_{\tau\in\R\Delta_d}\norm{s-\tau}_{\mathcal{L}^2}.$$
\begin{lemma}[Distance to the discriminant]\label{distance to the discriminant} Let $(L,c_L)$ be a real Hermitian ample line bundle over a real algebraic variety $X$ of dimension $n$. Let denote by $\norm{\cdot}_{h^d}$ the Hermitian metric on $L^d$ induced by a Hermitian metric $h$ on $L$ with positive curvature.  Then, as $d\rightarrow\infty$, we have the uniform estimate $$\mathrm{dist}_{\R\Delta_d}(s)=\min_{x\in\R X}\bigg(\frac{\norm{s(x)}^2_{h^d}}{d^n}+\frac{\norm{\nabla s(x)}^2_{h^d}}{d^{n+1}}\bigg)^{1/2}\big(\pi^{n/2}+O(1/d)\big).$$
\end{lemma} 
\begin{proof}
We will follow the approach of \cite[Section 3]{Raffalli} which we combine with Bergman kernel estimates along the diagonal \cite{daima,ma2,ber1,zel}. Let us define $\R \Delta_{d,x}$ to be the linear subspace of real sections which do not vanish transversally to $x$, that is 
$$\R \Delta_{d,x}=\{s\in\R H^0(X,L^d), s(x)=0\hspace{2mm} \textrm{and}\hspace{2mm} \nabla s(x)=0\}.$$   Let us fix some notations:
\begin{itemize}
\item We fix once for all an orthonormal basis $s_1,\dots,s_{N_d}$ of $\R H^0(X,L^d)$ (with respect to the $\mathcal{L}^2-$scalar product \eqref{l2 scalar}) and a trivialization $e_L$ of $L$ around $x$ whose norm equals $1$ at $x$.  
\item Any section $s_i$ of the basis can be written as $s_i=f_ie_L^d$.  We will denote by $f(x)$ the line vector $\big(f_1(x),\dots,f_{N_d}(x)\big)\in\R^{N_d}$. 
\item We will also fix an orthonormal basis $\{\frac{\partial}{\partial x_1},\dots,\frac{\partial}{\partial x_n}\}$ of $T_x\R X$ (the Riemannian metric is the one induced by the curvature form $\omega$ of the real Hermitian line bundle $(L,h)$). 
\item We denote by $\nabla_j$ the covariant derivative along $\frac{\partial}{\partial x_j}$ on $L^d$ induced by the metric $h^d$. Locally, the covariant derivative along $\frac{\partial}{\partial x_j}$ can be written as $\partial_j:=\frac{\partial}{\partial x_j}+\lambda_j$ (that is the sum of the derivative along $\frac{\partial}{\partial x_j}$ plus a term of order $0$).
\item  We will denote by $M$ the  $(n+1)\times N_d$ matrix whose $n+1$ lines are $f(x),\partial_1f(x),\dots, \partial_nf(x)$, where $\partial_jf(x)=(\partial_jf_1(x),\dots, \partial_jf_{N_d}(x))\in\R^{N_d}$.
\end{itemize}
 Using these notations, a real holomorphic section $s=\sum_{i}a_is_i$ lies in $\R \Delta_{d,x}$ if and only if $s(x)=0$ and $\nabla_{j}s(x)=0$ for any $j\in\{1,\dots,n\}$. Using the local picture, these conditions can be written as $\sum_{i=1}^{N_d}a_if_i(x)=0$ and  $\sum_{i=1}^{N_d}a_i\partial_jf_i(x)=0$ for any $j\in\{1,\dots,n\}$.
 
We  now compute $\textrm{dist}(s,\R \Delta_{d,x})$, for any $s\in\R H^0(X,L^d)$. By definition, the distance $\textrm{dist}(s,\R \Delta_{d,x})$ equals $\min_{s+\tau\in\R \Delta_{d,x}}\norm{\tau}_{\mathcal{L}^2},$ that is, it equals the $\mathcal{L}^2$-norm of a section $\tau$ which is orthogonal to $\R \Delta_{d,x}$ and such that $s+\tau\in\R \Delta_{d,x}$. Writing $s=\sum_{i}a_is_i$ and $\tau=\sum_ib_is_i$, the last condition reads $M (a+b)=0$, where $a=(a_1,\dots,a_{N_d})^t$, $b=(b_1,\dots,b_{N_d})^t$ and the multiplication is the standard matrix multiplication. The condition "$\tau$ is orthogonal to $\R \Delta_{d,x}$" can be written as "there exists $\alpha=(\alpha_0,\dots,\alpha_n)^t\in\R^{n+1}$ such that $b=M^t\alpha$". Putting together these two conditions, we obtain $M (a+M^t\alpha)=M a+M M^t\alpha=0$. Now, denoting $A=M M^t$, we have $\alpha=-A^{-1} M a$, so that $b=-M^tA^{-1}M a$.
We then obtain 
\begin{multline}\label{chaineofequalityinnorm}
\textrm{dist}^2(s,\R \Delta_{d,x})=\norm{\tau}^2_{\mathcal{L}^2}=\norm{b}^2=\norm{M^tA^{-1}M a}^2\\=a^tM^t(A^{-1})^tMM^tA^{-1}M a=a^tM^t(A^{-1})^tM a.
\end{multline}
  Remark now that $A$ is a $(n+1)\times(n+1)$-symmetric matrix of the following form
\[
\begin{bmatrix}
\langle f(x),f(x)\rangle & \langle \partial_1f(x),f(x)\rangle &\cdots & \langle \partial_nf(x),f(x)\rangle 
 \\
\langle f(x),\partial_1f(x)\rangle & \langle \partial_1f(x),\partial_1f(x)\rangle &\cdots & \langle \partial_nf(x),\partial_1f(x)\rangle 
 \\
 \vdots & \vdots & \ddots & \vdots 
 \\
\langle f(x),\partial_nf(x)\rangle & \langle \partial_1f(x),\partial_nf(x)\rangle &\cdots & \langle \partial_nf(x),\partial_nf(x)\rangle 
\end{bmatrix}
\]
where $\langle \cdot,\cdot\rangle$ denotes the standard scalar product on $\R^{N_d}$.

We can see that the quantities $\langle f(x),f(x)\rangle$,  $\langle \partial_if(x),f(x)\rangle$ and $\langle \partial_if(x),\partial_jf(x)\rangle$ equal the value of the Bergman kernel  and of its first derivatives at $(x,x)$. The asymptotics of these quantities are well-known (see, for instance, \cite{daima,ma2,ber1,zel}). In particular, we have $\langle f(x),f(x)\rangle=\frac{d^n}{\pi^n}+O(d^{n-1})$, $\langle \partial_if(x),f(x)\rangle=O(d^{n-1})$, $\langle \partial_if(x),\partial_jf(x)\rangle=O(d^n)$ for $i\neq j$ and  $\langle \partial_if(x),\partial_if(x)\rangle=\frac{d^{n+1}}{\pi^n}+O(d^{n})$, as $d\rightarrow\infty$.  As a consequence, as $d\rightarrow\infty$, the matrix $A$ equals
\[
\begin{bmatrix}
\frac{d^n}{\pi^n}+O(d^{n-1}) & O(d^{n-1})   &\cdots & O(d^{n-1}) 
 \\
O(d^{n-1})   & \frac{d^{n+1}}{\pi^n}+O(d^{n})  &\cdots & O(d^{n}) 
 \\
 \vdots & \vdots & \ddots & \vdots 
 \\
O(d^{n-1})  &  O(d^{n})  &\cdots &  \frac{d^{n+1}}{\pi^n}+O(d^{n})
\end{bmatrix}=
\begin{bmatrix}
\frac{d^n}{\pi^n} & 0  &\cdots &0
 \\
0  & \frac{d^{n+1}}{\pi^n} &\cdots & 0
 \\
 \vdots & \vdots & \ddots & \vdots 
 \\
0 &  0 &\cdots &  \frac{d^{n+1}}{\pi^n}
\end{bmatrix} \bigg(\mathrm{Id}+O(d^{-1})\bigg)
\]
so that $A^{-1}$ equals
\[A^{-1}=
\begin{bmatrix}
d^{-n} & 0  &\cdots &0
 \\
0  & d^{-n-1} &\cdots & 0
 \\
 \vdots & \vdots & \ddots & \vdots 
 \\
0 &  0 &\cdots &  d^{-n-1}
\end{bmatrix} \bigg(\pi^n\mathrm{Id}+O(d^{-1})\bigg).
\]
In particular, recalling the definition of the matrix $M$ and the vector $a$ given above, we get that \eqref{chaineofequalityinnorm} equals $\bigg(\frac{\norm{s(x)}^2_{h^d}}{d^n}+\frac{\norm{\nabla s(x)}^2_{h^d}}{d^{n+1}}\bigg)\big(\pi^n+O(1/d)\big)$. Taking the minimum over $x\in\R X$, we obtain the result.
\end{proof}

\begin{lemma}[Degree of the discriminant]\label{degree of discriminant} Let $(L,c_L)$ be a real ample line bundle over a real algebraic variety $X$ of dimension $n$ and denote by $\R\Delta_d$  the discriminant  in $\R H^0(X,L^d)$. Then, there exists an homogeneous real  polynomial $Q_d$ vanishing on $\R \Delta_d$ and such that ${\deg(Q_d)=(n+1)\int_{X}c_1(X)^nd^n+O(d^{n-1})}$.
\end{lemma}
\begin{proof}  Let $\Delta_d\subset H^0(X,L^d)$ be the (complex) discriminant. First, remark that if $s\in\R\Delta_d$, then $s\in\Delta_d$, so that, if we find a real polynomial vanishing on $\Delta_d$, then it will vanish also on $\R\Delta_d$. We will then find a polynomial vanishing along $\Delta_d$ and estimate its degree.  Remark that $\Delta_d$ is a cone  in $H^0(X,L^d)$, so that the degree of $\Delta_d$ equals the number of intersection points of a generic line $\gamma$ in $\mathbb{P}\big(H^0(X,L^d)\big)$  with $\mathbb{P}(\Delta_d)$. We remark that a line $\gamma$ in $\mathbb{P}\big(H^0(X,L^d)\big)$ induces a Lefschetz pencil $u:X\dashrightarrow \C P^1$, by sending $x\in X$ to $[s_1(x):s_2(x)]\in\C P^1$, where $s_1$ and $s_2$ are two distinct sections lying on $\gamma$. Now, the cardinality of the intersection $\gamma\cap \mathbb{P}(\Delta_d)$ corresponds exactly to the number of singular fibers of the Lefschetz pencil $u$. By \cite[Proposition 2.3]{anc}, we then have $\deg \big(\mathbb{P}(\Delta_d)\big)=(n+1)\int_{X}c_1(X)^nd^n+O(d^{n-1})$. In particular, there exists an homogenous polynomial $Q_d$ of degree $\deg\big(\mathbb{P}(\Delta_d)\big)$ vanishing on $\Delta_d$. 
Also,  if $s\in\Delta_d$, then we have $c_L\circ s\circ c_X\in\Delta_d$, so that $\Delta_d$ is a real algebraic hypersurface in $H^0(X,L^d)$ (with respect to the natural real structure $s\mapsto c_{L^d}\circ s\circ c_X$). It is then possible to choose the polynomial $Q_d$ to be real.
\end{proof}

\begin{defn}\label{tubular} A \textit{tubular conical neighborhood} is a tubular neighborhood which is a cone (that is, if $s$ is in the neighborhood, then $\lambda s$ is also in the neighborhood, for any $\lambda\in\R^*$).
\end{defn}
 To define a tubular conical neighborhood of $\R\Delta_d$, it is then enough to define its trace on the unit sphere $S_d:=\{s\in\R H^0(X,L^d),\norm{s}_{\mathcal{L}^2}=1\}$. We  also denote by $S\Delta_d$ the trace of the discriminant on $S_d$, that is $S\Delta_d=S_d\cap \R\Delta_d$.
  The next lemma  estimates the measure of \textit{very small} tubular conical neighborhoods of the discriminant $\R\Delta_d$, that are tubular conical neighborhoods whose trace on $S_d$ is of the form $\{s\in S_d, \mathrm{dist}(s,S\Delta_d)< r_d\}$, for $r_d$ a sequence of real numbers which goes to zero \textit{at least} as $d^{-2n}$.
  % It is  an application in our context of \cite[Theorem 21.1]{burcondition}  (see also \cite[Proposition 4]{diatta}).
\begin{lemma}[Volume of tubular conical neighborhoods] \label{volume of tubular neighborhood} Let $(L,c_L)$ be a real Hermitian ample line bundle over a real algebraic variety $(X,c_X)$ of dimension $n$. Then there exists a positive constant $c$, such that, for any sequence $r_d$ verifying $r_d\leq cd^{-2n}$, one has 
$$\mu_d\{s\in\R H^0(X,L^d), \mathrm{dist}_{\R\Delta_d}(s)\leq r_d\norm{s}_{\mathcal{L}^2}\}\leq O(r_dd^{2n}).$$
Here, $\mu_d$ is the Gaussian probability measure defined in Equation \eqref{gaussian measure}.
\end{lemma}
\begin{proof}
We start by a standard remark about Gaussian measures on Euclidian spaces: let us denote by $S_d$ the unit sphere in $\R H^0(X,L^d)$, and by $\nu_d$ the probability measure induced by its volume form (i.e. for any $A\subset S_d$, $\nu_d(A)=\Vol(A)\Vol(S_d)^{-1}$), then the Gaussian measure of every cone $C_d$ in $\R H^0(X,L^d)$ equals $\nu_d(C_d\cap S_d)$. We apply this remark to the tubular conical neighborhood $C_d=\{\mathrm{dist}_{\R\Delta_d}(s)\leq r_d\norm{s}_{\mathcal{L}^2}\}$ and obtain that its Gaussian measure equals $\nu_d\{s\in S_d, \mathrm{dist}(s,S\Delta_d)\leq r_d\}$, where  $S\Delta_d$ denote the trace of the discriminant on $S_d$. By Lemma \ref{degree of discriminant}, we have that, for $d$ large enough, there exists a polynomial of degree bounded by $2(n+1)\int_{X}c_1(X)^nd^n$  whose zero locus contains $S\Delta_d$. We are then in the hypothesis of \cite[Theorem 21.1]{burcondition} which gives us
\begin{equation}\label{numeasure}
\nu_d\{s\in S_d, \mathrm{dist}(s,S\Delta_d)\leq r_d\}\leq cN_dd^nr_d
\end{equation}
for some constant $c>0$, where $N_d=\dim \R H^0(X,L^d)$. By Riemann-Roch Theorem, we have that $N_d=O(d^n)$, so that the right-hand side of \eqref{numeasure} is $O(r_dd^{2n})$.
The lemma then follows from the fact that the $\nu_d$-measure of the set $\{s\in S_d, \mathrm{dist}(s,S\Delta_d)\leq r_d\}$ appearing in \eqref{numeasure} equals the Gaussian measure of its cone. This implies 
$$\mu_d\{s\in\R H^0(X,L^d), \mathrm{dist}_{\R\Delta_d}(s)\leq r_d\norm{s}_{\mathcal{L}^2}\}\leq O(r_dd^{2n}),$$
which concludes the proof.
\end{proof}

\section{Proof of the main results}\label{SecProof}
In this section, we prove our main results, namely Theorems \ref{theorem rarefaction}, \ref{theorem exponential rarefaction} and \ref{theorem approximation}. 
We will use the notations of the previous sections, in particular we consider an ample Hermitian real holomorphic line bundle $(L,c_L,h)$ over a real algebraic variety $(X,c_X)$ and we denote by $\mu_d$ the Gaussian measure on $\R H^0(X,L^d)$ defined in Equation \eqref{gaussian measure}. 
\begin{conv}\label{orthogonal component}  Let $\sigma\in \R H^0(X,L^k)$ be a section given by Proposition \ref{existance of sigma}, for some fixed $k$ large enough. Let $d$ and $\ell$ be two positive integers  and let $\R H_{d,\sigma^\ell}$ be as in Definition \ref{definitionvanishing}. For any real section $s\in\R H^0(X,L^d)$ there exists an unique decomposition $s=s_{\sigma^{\ell}}^{\perp}+s_{\sigma^{\ell}}^{0}$ with $s_{\sigma^{\ell}}^{0}\in \R H_{d,\sigma^\ell}$ and $s_{\sigma^{\ell}}^{\perp}\in \R H^{\perp}_{d,\sigma^\ell}$  (the orthogonal is with respect to the $\mathcal{L}^2$-scalar product defined in Equation \eqref{l2 scalar}). 
\end{conv}
\begin{prop}\label{estimatepartial} Let $\sigma\in \R H^0(X,L^k)$ be a section given by Proposition \ref{existance of sigma}, for some fixed $k$ large enough. The following holds.
\begin{enumerate}
 \item There exists a positive real number $t_0$ such that for any $t\in(0,t_0)$ the following happens. Let $C>0$ and $r\in\mathbb{N}$. For any $w_d\geq Cd^{-r}$, there exists  $d_0\in\mathbb{N}$, such that for any $d\geq d_0$, we have  $$\mu_d\{s\in\R H^0(X,L^d), \norm{s_{\sigma^{\lfloor t d\rfloor}}^{\perp}}_{\mathcal{C}^1(\R X)}\geq w_d\norm{s}_{\mathcal{L}^2}\}=0.$$
 \item There exist two positive constants  $c_1$ and $c_2$ such that, for any sequence $w_d\geq c_1e^{-c_2\sqrt{d}\log d}$, we have 
  $$\mu_d\{s\in\R H^0(X,L^d), \norm{s_{\sigma}^{\perp}}_{\mathcal{C}^1(\R X)}\geq w_d\norm{s}_{\mathcal{L}^2}\}=0.$$
  If the real Hermitian metric on $L$ is  analytic, then the last estimate is true for any sequence $w_d\geq c_1e^{-c_2d}$.
\end{enumerate} 
Here, $s_{\sigma}^{\perp}$ and $s_{\sigma^{\lfloor t d\rfloor}}^{\perp}$ are given by Notation \ref{orthogonal component}. 
\end{prop}
\begin{proof}
Let us prove point \textit{(1)}. Let us fix $t_0$ given by Proposition \ref{partialbergman} and take $t<t_0$.
Let us denote by $S_d$ the unit sphere in $\R H^0(X,L^d)$, and by $\nu_d$ the probability measure induced by its volume form (i.e. for any $A\subset S_d$, $\nu_d(A)=\Vol(A)\Vol(S_d)^{-1}$). Then, the Gaussian measure of every cone $C_d$ in $\R H^0(X,L^d)$ equals $\nu_d(C_d\cap S_d)$. We apply this remark to the cone we are interested in, that is $C_d=\big\{s\in\R H^0(X,L^d), \norm{s_{\sigma^{\lfloor t d\rfloor}}^{\perp}}_{\mathcal{C}^1(\R X)}\geq w_d\norm{s}_{\mathcal{L}^2}\big\}$, and obtain 
\begin{equation}\label{comparaisonofmeas}
\mu_d\big\{s\in\R H^0(X,L^d), \norm{s_{\sigma^{\lfloor t d\rfloor}}^{\perp}}_{\mathcal{C}^1(\R X)}\geq w_d\norm{s}_{\mathcal{L}^2}\big\}=\nu_d\big\{s\in S_d, \norm{s_{\sigma^{\lfloor t d\rfloor}}^{\perp}}_{\mathcal{C}^1(\R X)}\geq w_d\big\}.
\end{equation}
Now, by Proposition \ref{partialbergman}, as $\norm{s^{\perp}_{\sigma^{\lfloor t d\rfloor}}}_{\mathcal{L}^2}\leq 1$, there exists a  constant $c_r>0$ such that $\norm{s_{\sigma^{\lfloor t d\rfloor}}^{\perp}}_{\mathcal{C}^1(\R X)}\leq c_r d^{-r-1}$, which is strictly smaller than $w_d$, for $d$ large enough. We then obtain $$\nu_d\big\{s\in S_d, \norm{s_{\sigma^{\lfloor t d\rfloor}}^{\perp}}_{\mathcal{C}^1(\R X)}\geq w_d\big\}=0,$$ which, together with \eqref{comparaisonofmeas}, proves  point \textit{(1)} of the result.
 The proof of  point \textit{(2)} follows the same lines, using Proposition \ref{logberg} instead of Proposition \ref{partialbergman}.
\end{proof}

\begin{lemma}\label{isotopy}  There exists a positive integer $d_0$ such that for any $d\geq d_0$ and any real section $s\in\R H^0(X,L^d)\setminus \R\Delta_d$, the following happens. For any real global section $s'\in\R H^0(X,L^d)$ such that $$\norm{s-s'}_{\mathcal{C}^1(\R X)}< \frac{d^{n/2}}{4\pi^{n/2}}\mathrm{dist}_{\R\Delta_d}(s),$$ we have that the pairs $(\R X,\R Z_s)$ and $(\R X,\R Z_{s'})$ are isotopic.
\end{lemma} 
\begin{proof}
By Lemma \ref{distance to the discriminant}, there exists $d_0$ such that for any $d\geq d_0$ one has  $$\mathrm{dist}_{\R\Delta_d}(s)<2\pi^{n/2}\min_{x\in\R X}\bigg(\frac{\norm{s(x)}^2_{h^d}}{d^n}+\frac{\norm{\nabla s(x)}^2_{h^d}}{d^{n+1}}\bigg)^{1/2}.$$ In particular, for any  $d\geq d_0$, the inequality  $\norm{s-s'}_{\mathcal{C}^1(\R X)}< \frac{d^{n/2}}{4\pi^{n/2}}\mathrm{dist}_{\R\Delta_d}(s)$ implies
\begin{equation}\label{equationindistance}
\norm{s-s'}_{\mathcal{C}^1(\R X)}< \frac{1}{2}\min_{x\in\R X}\bigg(\norm{s(x)}^2_{h^d}+\frac{\norm{\nabla s(x)}^2_{h^d}}{d}\bigg)^{1/2}.
\end{equation} 
Now, denoting $\delta(s):=\min_{x\in\R X}\bigg(\norm{s(x)}^2_{h^d}+\frac{\norm{\nabla s(x)}^2_{h^d}}{d}\bigg)^{1/2}$ and following the lines of \cite[Proposition 3]{diatta}, we have that the inequality \eqref{equationindistance} implies that the pairs $(\R X,\R Z_s)$ and $(\R X,\R Z_{s'})$ are isotopic.
\end{proof}
\begin{lemma}\label{rapiddecay} Let $\sigma\in \R H^0(X,L^k)$ be a section given by Proposition \ref{existance of sigma}, for some fixed $k$ large enough.  Then, we have the following estimates as $d\rightarrow\infty$.
\begin{enumerate}
\item There exists $t_0>0$ such that, for any $t\in (0,t_0)$,  we have $$\mu_d\big\{s\in\R H^0(X,L^d), \norm{s_{\sigma^{\lfloor t d\rfloor}}^{\perp}}_{\mathcal{C}^1(\R X)}<\frac{d^{n/2}}{4\pi^{n/2}}\mathrm{dist}_{\R\Delta_d}(s)\big\}\geq 1-O(d^{-\infty}).$$
\item There exists a positive $c>0$ such that$$\mu_d\big\{s\in\R H^0(X,L^d), \norm{s_{\sigma}^{\perp}}_{\mathcal{C}^1(\R X)}<\frac{d^{n/2}}{4\pi^{n/2}}\mathrm{dist}_{\R\Delta_d}(s)\big\}\geq 1-O(e^{-c\sqrt{d}\log d}).$$ 
Moreover, if the real Hermitian metric on $L$ is  analytic, then  the last measure is even bigger then $1-O(e^{-cd})$.
\end{enumerate}
 
\end{lemma}
\begin{proof}
First, remark that, by Proposition \ref{volume of tubular neighborhood}, for any $m\in\mathbb{N}$, setting $r_d=C_1d^{-2n-m}$, one has 
\begin{equation}\label{rapid1}
\mu_d\big\{s\in\R H^0(X,L^d), \mathrm{dist}_{\R\Delta_d}(s)> r_d\norm{s}_{\mathcal{L}^2}\big\}\geq 1-O(d^{-m}).
\end{equation}
Also, by  point \textit{(1)} of Proposition \ref{estimatepartial}, for any $t<t_0$, any integer $r$,   any sequence $w_d$ of the form $C_2d^{-r}$ and any  $d$ large enough, we have
\begin{equation}\label{rapid2}
\mu_d\big\{s\in\R H^0(X,L^d), \norm{s_{\sigma^{\lfloor t d\rfloor}}^{\perp}}_{\mathcal{C}^1(\R X)}< w_d\norm{s}_{\mathcal{L}^2}\big\}= 1.
\end{equation}
Putting together \eqref{rapid1} and \eqref{rapid2}, we have that, for any such sequences $r_d$ and $w_d$,
\begin{equation}\label{rapid3}
\mu_d\big\{s\in\R H^0(X,L^d), \norm{s_{\sigma^{\lfloor t d\rfloor}}^{\perp}}_{\mathcal{C}^1(\R X)}< \frac{w_d}{r_d}\mathrm{dist}_{\R\Delta_d}(s)\big\}\geq 1-O(d^{-m}).
\end{equation}
By choosing $w_d=\frac{d^{n/2}}{4\pi^{n/2}}r_d$, we then obtain that for any $m\in\mathbb{N}$
\begin{equation}\label{rapid3}
\mu_d\big\{s\in\R H^0(X,L^d), \norm{s_{\sigma^{\lfloor t d\rfloor}}^{\perp}}_{\mathcal{C}^1(\R X)}< \frac{d^{n/2}}{4\pi^{n/2}}\mathrm{dist}_{\R\Delta_d}(s)\big\}\geq 1-O(d^{-m})
\end{equation}
which proves the point \textit{(1)} of the proposition. Point \textit{(2)} of the proposition follows the same lines, using  point \textit{(2)} of Proposition \ref{estimatepartial} and   setting  $r_d=d^{-2n}e^{-c\sqrt{d}}$, where $c$ is given by point \textit{(2)} of Proposition \ref{estimatepartial}.
\end{proof}
We now prove Theorems \ref{theorem approximation}, \ref{theorem rarefaction} and \ref{theorem exponential rarefaction}.
\begin{proof}[Proof of Theorem \ref{theorem approximation}]
Let us start with the proof of point \textit{(1)} of the theorem. \\
Let $\sigma\in \R H^0(X,L^k)$ be a section given by Proposition \ref{existance of sigma}, for some fixed $k$ large enough.
We want to prove that there exists $b_0<1$ such that for any $b> b_0$, the measure 
\begin{equation}\label{measuretoestimate}
\mu_d\big\{s\in\R H^0(X,L^d), \exists\hspace{0.5mm} s'\in\R H^0(X,L^{\lfloor bd\rfloor}) \hspace{1mm} \textrm{such that} \hspace{1mm} (\R X,\R Z_s)\sim (\R X,\R Z_{s'}) \big\}
\end{equation}
equals $1-O(d^{-\infty})$, as $d\rightarrow\infty$, where $(\R X,\R Z_s)\sim (\R X,\R Z_{s'})$ means there the two pairs are isotopic. Let us consider $t_0>0$ given by point \textit{(1)} of Proposition \ref{rapiddecay} and set $b_0=1-kt_0$. Then, for any  $b> b_0,$ there exists $t< t_0$, such that  $\lfloor bd\rfloor\geq d-k\lfloor t d\rfloor$ holds. Let us fix such $b$ and $t$.  For any $s\in \R H^0(X,L^d)$, let us write the orthogonal decomposition $s=s_{\sigma^{\lfloor t d\rfloor}}^{\perp}+s_{\sigma^{\lfloor t d\rfloor}}^{0}$ given by Notation \ref{orthogonal component}.
By Lemma \ref{isotopy}, if the $\mathcal{C}^1(\R X)$-norm of $s_{\sigma^{\lfloor t d\rfloor}}$ is smaller than $\frac{d^{n/2}}{4\pi^{n/2}}\mathrm{dist}_{\R\Delta_d}(s)$, then the pairs $(\R X,\R Z_s)$ and $(\R X,\R Z_{s_{\sigma^{\lfloor t d\rfloor}}^{0}})$ are isotopic. This implies that the measure \eqref{measuretoestimate} is bigger than the Gaussian measure of the set
\begin{equation}\label{measuretoestimate2}
\bigg\{s\in\R H^0(X,L^d), \norm{s_{\sigma^{\lfloor t d\rfloor}}^{\perp}}_{\mathcal{C}^1(\R X)}<\frac{d^{n/2}}{4\pi^{n/2}}\mathrm{dist}_{\R\Delta_d}(s) \bigg\}
\end{equation}
which, in turn, by the point \textit{(1)} of Proposition \ref{rapiddecay}, is bigger than $1-O(d^{-\infty})$. We have then proved that the pair $(\R X,\R Z_s)$ is isotopic to the pair $(\R X,\R Z_{s_{\sigma^{\lfloor t d\rfloor}}^{0}})$ with probability $1-O(d^{-\infty})$. Now, $s_{\sigma^{\lfloor t d\rfloor}}^{0}\in \R H_{d,\sigma^{\lfloor t d\rfloor}}$ which implies, by Proposition \ref{vanishing order}, that there exists $s'\in \R H^0(X,L^{d-k\lfloor t d\rfloor})$ such that $s_{\sigma^{\lfloor t d\rfloor}}^{0}=\sigma^{\lfloor t d\rfloor}\otimes s'$. Assertion \textit{(1)} of the theorem then follows from the fact that  the real zero locus of $s_{\sigma^{\lfloor t d\rfloor}}^{0}$ coincides with the real zero locus of $s'$. Indeed,  the real zero locus of $s_{\sigma^{\lfloor t d\rfloor}}^{0}$ equals $\R Z_\sigma\cup\R Z_{s'}$ and this is equal to $\R Z_{s'}$, because  $\R Z_\sigma=\emptyset$.

Assertion \textit{(2)} of the theorem follows the same lines,  using the orthogonal decomposition  $s=s_{\sigma}^{\perp}+s_{\sigma}^{0}$ and  Assertion \textit{(2)} of Proposition \ref{rapiddecay}.
\end{proof}

\begin{proof}[Proof of Theorems \ref{theorem rarefaction}]
Recall that we want to prove that there exists $a_0<v(L)$ such that, for any   $a>a_0$, we have 
$$\mu_d\{s\in\R H^0(X,L^d), b_*(\R Z_s) < ad^n\}=1-O(d^{-\infty})$$
as $d\rightarrow\infty$.
Let $a_0:=\int_Xc_1(L)^n b_0^n$, where $b_0$ is given by point \textit{(1)} of Theorem \ref{theorem approximation} and let $a>a_0$. 

By point \textit{(1)} of Theorem \ref{theorem approximation}, for any $b>b_0$, the real zero locus of a global section $s$ is diffeomorphic to the real zero locus of a global section $s'$ of $L^{\lfloor bd\rfloor}$  with probability $1-O(d^{-\infty})$. Now, by Smith-Thom inequality (see Equation \eqref{smith-thom}), the total Betti number $b_*(\R Z_{s'})$ of the real zero locus a generic section $s'$ of $L^{\lfloor bd\rfloor}$ is smaller or equal than $b_*(Z_{s'})$, which, by \cite[Lemma 3]{gw1}, has the asymptotic $b_*(Z_{s'})= \int_Xc_1(L)^n\big(\lfloor bd\rfloor\big)^n+O(d^{n-1})$. In particular, with probability $1-O(d^{-\infty})$, the  total Betti number $b_*(\R Z_{s})$ of the real zero locus of a section $s$ of $L^d$ is smaller  than $\int_Xc_1(L)^n\big(\lfloor bd\rfloor\big)^n+\epsilon d^{n}$, for any $\epsilon>0$, as $d\rightarrow\infty$. Choosing $\epsilon$ small enough, we can find $1>b> b_0$ such that $\int_Xc_1(L)^n b^n+\epsilon<a$, which implies  the result.
\end{proof}
\begin{proof}[Proof of Theorem \ref{theorem exponential rarefaction}]
We want to prove that, for any  $a>0$, there exists $c>0$ such that 
$$\mu_d\big\{s\in\R H^0(X,L^d), b_*(\R Z_s)<b_*(Z_s)-ad^{n-1}\big\}\geq 1- O(e^{-c\sqrt{d}\log d})$$
 as $d\rightarrow\infty$ (and also the similar estimate with $1-O(e^{-cd})$ in the right-hand side, if the real Hermitian metric on $L$ is  analytic).

Remark that the total Betti number of the zero locus $Z_s$ a generic section $s$ of $L^{d}$ has the asymptotic 
\begin{equation}\label{tsd}
b_*(Z_s)=\int_{X}c_1(L)^nd^{n}-\int_{X}c_1(L)^{n-1}\wedge c_{1}(X)d^{n-1}+O(d^{n-2}),
\end{equation}
 as $d\rightarrow\infty$, where $c_1(X)$ is the first Chern class of $X$. (This is a classical estimate which follows from the adjunction formula and Lefschetz hyperplane section theorem. A  proof of this estimates  follows the lines of \cite[Lemma 3]{gw1}, by just taking care of the second term of the asymptotic which can be derived using \cite[Lemma 2]{gw1}).\\
Similarly, for any $k\in\mathbb{N}$, the zero locus of a generic section $s'$ of $L^{d-k}$ is such that 
\begin{equation}\label{tsdk}
b_*(Z_{s'})=\int_{X}c_1(L)^nd^{n}-\bigg(\int_{X}c_1(L)^{n-1}\wedge c_{1}(X)+nk\int_{X}c_1(L)^n\bigg)d^{n-1}+O(d^{n-2}),
\end{equation}
 as $d\rightarrow\infty$. Let us choose once for all an integer $k$ big enough such that $nk\int_{X}c_1(L)^n>a$. In particular,  by \eqref{tsd} and \eqref{tsdk}, the zero loci of generic sections $s$ and $s'$ of $L^d$ and $L^{d-k}$ are such that 
\begin{equation}\label{bettiinequality}
b_*(Z_{s'})<b_*(Z_s)-ad^{n-1},
\end{equation}
$d\rightarrow\infty$.
Now, by the Assertion \textit{(2)} of Theorem \ref{theorem approximation}, the probability that the real zero locus of a section $s$ of $L^d$ is diffeomorphic to the real zero locus of a section $s'$ of $L^{d-k}$ equals $1-O(e^{-c\sqrt{d}\log d})$ (or $1-O(e^{-cd})$, if the real Hermitian metric on $L$ is  analytic). By Smith-Thom inequality and by  \eqref{bettiinequality}, this implies that the probability of the event "the total Betti number of the real zero locus of a section $s$ of $L^d$ is smaller  than $b_*(Z_s)-ad^{n-1}$" is bigger or equal than $1-O(e^{-c\sqrt{d}\log d})$ (or $1-O(e^{-cd})$, if the real Hermitian metric on $L$ is  analytic), which proves the theorem.
\end{proof}
\section{Real curves in algebraic surfaces and depth of nests}\label{secnest}
In this section, we study the nests defined by a random real algebraic curve inside a real algebraic surface. We will focus in particular on two examples: the Hirzebruch surfaces (Section \ref{secHirze}) and the Del Pezzo surfaces (Section \ref{secDelPezzo}). The main results of the section (Theorems \ref{hirzethm} and \ref{delpezzothm})  say that, inside these surfaces, real algebraic curves with deep nests are rare, in the same spirit of Theorems \ref{theorem rarefaction} and \ref{theorem exponential rarefaction}.
In order to define what "deep nests" means in this context, we need to prove some bounds on the depth of nests (Propositions \ref{nesthirze} and \ref{nestdelpezzo}) which are well known to experts. These bounds play the same role that Smith-Thom inequality \eqref{smith-thom} plays for Theorems \ref{theorem rarefaction} and \ref{theorem exponential rarefaction}.
Theorems \ref{hirzethm} and \ref{delpezzothm} will also follows from our low degree approximation (Theorem \ref{theorem approximation}). 
\begin{defn} Let $S$ be a real algebraic surface and let $M$ be a connected component of the real part $\R S$. 
\begin{itemize}
\item An \textit{oval} in $M$ is an embedded circle $\mathbb{S}^1\hookrightarrow M$ which bounds a disk.  
\item Suppose that $M$ is not diffeomorphic to the sphere $\mathbb{S}^2$, then  two ovals form an \textit{injective pair} if one of the two oval is contained in the disk bounded by the other one. A \textit{nest} $N$ in $M$ is a collection of ovals with the property that each pair of ovals in the collection is an injective pair.

\item Suppose that  $M$ is diffeomorphic to  $\mathbb{S}^2$, then a \textit{nest} $N$ is a collection of ovals with the property that each connected component of $M\setminus N$ is either a disk or an annulus.

\item The number of ovals which form a nest is called \textit{depth} of the nest.

\end{itemize}
\end{defn}

\subsection{Nests in Hirzebruch surfaces}\label{secHirze} A \textit{Hirzebruch surface} $S$  is a compact complex surface which admits a holomorphic fibration over $\C P^1$ with fiber $\C P^1$. Given a Hirzebruch surface $S$, then there exists a positive integer $n$ such that $S$ is isomorphic to the projective bundle $P\big(\mathcal{O}_{\C P^1}(n)\oplus \mathcal{O}_{\C P^1}\big)$ over $\C P^1$ (see for example \cite{beauvilleSurface}). We will denote such a Hirzebruch surface by $\Sigma_n$.
Remark that a Hirzebruch surface is simply connected, that $\Sigma_0$ is isomorphic to $\C P^1\times\C P^1$ and $\Sigma_1$  to $\C P^2$ blown-up at a point. 

The Hirzebruch surface $\Sigma_n$ admits a natural real structure $c_n$, which we wix from now on, such that $\R\Sigma_n$ is diffeomorphic to a torus if $n$ is even and to a Klein bottle if $n$ is odd (see, for instance, \cite[Page  7]{welthese}).

 We denote by $F_n$ any fiber of the natural fibration $\Sigma_n\rightarrow \C P^1$  and by $B_n$ the section $P(\mathcal{O}_{\C P^1}(n)\oplus \{0\})$ of this fibration. Remark that $B_n$ is a real algebraic curve and $F_n$ can be chosen to be a real algebraic curve.
 The homology classes $[F_n]$ and $[B_n]$ form a basis of $H_2(\Sigma_n,\mathbb{Z})$, with $[F_n]\cdot [F_n]=0$, $[F_n]\cdot [B_n]=1$ and $[B_n]\cdot [B_n]=n$, where "$\cdot$" denotes the intersection product of $H_2(\Sigma_n,\mathbb{Z})$. 
 
 The second homology group $H_2(\Sigma_n,\mathbb{Z})$ can be identified with the Picard group of $\Sigma_n$. We say the an algebraic curve in $\Sigma_n$ realizes the class $(a,b)$ (or it is of bidegree $(a,b)$) if its homology class equals $a[F_n]+b[B_n]$. If $a,b>0$, then the divisor associated with the class $(a,b)$ is an ample divisor.
 \begin{prop}\label{nesthirze} Let $(\Sigma_n,c_n)$ be a real Hirzebruch surface and let $a,b$ be two positive integers.  Then, the depth of a nest of a real algebraic curve in the class $(a,b)$ is smaller or equal than $ b/2$.
 \end{prop}
  \begin{proof}  
 Let $C$ be a real curve of bidegree $(a,b)$, with $a,b>0$. Let $N$ be a nest of $\R C$ of depth $l$ and choose a point $p$ inside the innermost oval and a point $q$ outside the outer oval of the nest. We can choose $p$ and $q$ so that they lie in the same fiber $F_n$ of the fibration $\Sigma_n\rightarrow \C P^1$ and such that real locus of the fiber $\R F_n$ intersects the nest transversally. The real locus of the fiber $\R F_n$ intersects each of the $l$ ovals of the nest in at least $2$ points, so that $\R F_n\cap N\geq 2l$. On the other hand, by Bezout's theorem, we have that $\R F_n\cap N\leq [F_n]\cdot [C]\leq  b$, hence the result.
 \end{proof}
 
 \begin{thm}\label{hirzethm} Let $(\Sigma_n,c_n)$ be a real Hirzebruch surface and  $a$ and $b$ be two positive integers. 
\begin{enumerate} 
\item There exists a positive $\beta_0<\lfloor b/2\rfloor$ such that for any $\beta\in (\beta_0,\lfloor b/2\rfloor)$,  the probability that a real algebraic curve of bidegree $(da,db)$ has a nest of depth $\geq \beta d$ is smaller or equal than $O(d^{-\infty})$.
\item For any $k\in\mathbb{N}$ there exists a positive constant $c$ such that the probability that  a real algebraic curve of bidegree $(da,db)$ has a nest of depth $\geq \lfloor db/2\rfloor-k$ is smaller or equal than $O(e^{-c\sqrt{d}\log d})$.
 \end{enumerate}
 \end{thm}
 \begin{proof}
Notice that as a probability space we are considering the linear system associated with the divisor of bidegree $(a,b)$, rather than the space of real sections of the associated line bundle. As noticed in  Remark \ref{fubinistudy}, these two probability spaces are equivalent for our purpose.

  By Assertion \textit{(1)} of Theorem \ref{theorem approximation}, the real  locus of a real algebraic curve of bidegree $(da,db)$   is ambient isotopic to the real  locus of a real algebraic curve of bidegree $\big(\lfloor dta\rfloor, \lfloor dtb \rfloor\big)$ with probability bigger or equal than $1-O(d^{-\infty})$, where $t_0<t<1$ for some $t_0$.  From this, using Proposition \ref{nesthirze}, we obtain that with probability bigger or equal than $1-O(d^{-\infty})$, the depth of a nest of a real algebraic curve of bidegree $(da,db)$  is smaller or equal than $\lfloor dtb/2\rfloor$. Setting $\beta_0=\lfloor t_0b/2\rfloor$, we obtain Assertion \textit{(1)}.
 
 The proof  of the Assertion \textit{(2)} follows the same  lines, by   using Assertion \textit{(2)} of Theorem \ref{theorem approximation} instead of Assertion \textit{(1)}.
 \end{proof}

 \subsection{Nests in Del Pezzo surfaces}\label{secDelPezzo} An algebraic surface $S$ is called a \textit{Del Pezzo surface}, if the anticanonical bundle $K_S^*$ is ample. We will say that a real algebraic curve inside $S$ is of class $d$ if it belongs to the real linear system defined by the real  divisor $-dK_S$, where $d\in\mathbb{N}$.
 
 The \textit{degree} of a Del Pezzo surface is the self-intersection of the canonical class $K_S$.
  \begin{conv} We denote by $m_S$  the smallest value of $-H\cdot K_S$, where $H$ is a real divisor   such that $-H\cdot K_S-1\geq 2$. 
  \end{conv}
  When $S$ has degree bigger or equal than $3$, choosing $H=-K_S$, we can see that $m_S$ is smaller or equal than the degree of $S$.
 \begin{prop}\label{nestdelpezzo}
 Let  $(S,c_S)$ be a real Del Pezzo surface of degree bigger or equal than $3$, with non empty real locus.  Then, the depth of a nest of a real algebraic curve of class $d$ is smaller or equal than $dm_S/2$.
 \end{prop} 
\begin{proof}
 Let $m_S$ be the smallest value of $-H\cdot K_S$, where $H$ is a real divisor  with the property that $-H\cdot K_S-1\geq 2$. Fix such a real divisor $H$.
 
 Let $Z$ a real  algebraic curve of class $d$ and let $N$ be a nest of $\R Z$ of depth $l$. Choose generic points $p$ and $q$ such that  $p$ is inside the innermost oval and   $q$ is outside the outer oval of the nest (if the connected component $M$ where the nest is located is a sphere, then $M\setminus N$ as exactly two disks, and  one of the ovals bordering one of these two disks is the innermost oval and the other is the outer one). By \cite{itenbergDelPezzo}, there exists a real rational curve $C$ of class $H$  whose real locus passes through $p$ and $q$.  This curve intersects each oval of the nest in at least two points, so that $N\cap \R C\geq 2l$.  
  On the other hand, by Bezout's theorem, we have that $\R C\cap N\leq -d K_s\cdot H=dm_S$, hence the result.
\end{proof}
\begin{oss} The same statement and proof also works for many degree $2$ Del Pezzo surfaces (always using a real rational curves passing through $2$ points, whose existence is ensured by the positivity of certain Welschinger invariants  \cite{itenbergDelPezzo}). However, there exist  degree $2$ Del Pezzo surfaces for which such curves may not exist, as explained in \cite[Section 2.3]{itenbergDelPezzo}. To avoid explaining these differences, we have preferred to restrict ourselves only to Del Pezzo surfaces of degree greater than or equal to $3$.
\end{oss}
\begin{thm}\label{delpezzothm}
 Let $(S,c_S)$ be a  Del Pezzo surface of degree bigger or equal than $3$,  with non empty real locus.  Then the following happens.
 \begin{enumerate} 
\item There exists a positive $\alpha_0<\lfloor m_S/2\rfloor$ such that for any $\alpha\in (\alpha_0,\lfloor m_S/2\rfloor)$,  the probability that a real curve of class $d$  has a nest of depth $\geq \alpha d$ is smaller or equal than $O(d^{-\infty})$.
\item For any $k\in\mathbb{N}$ there exists a positive constant $c$ such that the probability that a real curve of class $d$ has a nest of depth $\geq \lfloor dm_S/2\rfloor-k$ is smaller or equal than $O(e^{-c\sqrt{d}\log d})$.
 \end{enumerate}
\end{thm} 
 \begin{oss} The degree of a Del Pezzo surface is at least $1$ and at most $9$. The only degree $9$ Del Pezzo surface is the projective plane $\C P^2$ and, in this particular case,  Theorem \ref{delpezzothm} follows from \cite[Theorem 9]{diatta}.
 \end{oss}
\begin{proof}[Proof of Theorem \ref{delpezzothm}]
 The proof  of the follows the same  lines of the proof of Theorem \ref{hirzethm}, using Proposition \ref{nestdelpezzo} instead of Proposition \ref{nesthirze}.
\end{proof}
\section*{Acknowledgments} I thank Antonio Lerario for stimulating questions which gave rise to this paper and for his comments that have improved the structure of the article. I also thank Thomas Letendre for his careful reading of the text, Matilde Manzaroli for fruitful discussions on nests of real algebraic curves  and Jean-Yves Welschinger for his interest in this work.
\bibliographystyle{plain}
\bibliography{biblio}

\begin{thebibliography}{10}

\bibitem{anc}
Michele Ancona.
\newblock Expected number and distribution of critical points of real
  {L}efschetz pencils.
\newblock {\em To appear at Annales de l'Institut Fourier}, arXiv:1707.08490.

\bibitem{beauvilleSurface}
Arnaud Beauville.
\newblock {\em Complex algebraic surfaces}, volume~34 of {\em London
  Mathematical Society Student Texts}.
\newblock Cambridge University Press, Cambridge, second edition, 1996.
\newblock Translated from the 1978 French original by R. Barlow, with
  assistance from N. I. Shepherd-Barron and M. Reid.

\bibitem{ber1}
Robert Berman, Bo~Berndtsson, and Johannes Sj\"ostrand.
\newblock A direct approach to {B}ergman kernel asymptotics for positive line
  bundles.
\newblock {\em Ark. Mat.}, 46(2):197--217, 2008.

\bibitem{bertrand}
Benoit Bertrand.
\newblock Asymptotically maximal families of hypersurfaces in toric varieties.
\newblock {\em Geom. Dedicata}, 118:49--70, 2006.

\bibitem{breiding}
Paul Breiding, Hanieh Keneshlou, and Antonio Lerario.
\newblock Quantitative singularity theory for random polynomials, 2019,
  arXiV:1909.11052.

\bibitem{burcondition}
Peter B\"{u}rgisser and Felipe Cucker.
\newblock {\em Condition}, volume 349 of {\em Grundlehren der Mathematischen
  Wissenschaften [Fundamental Principles of Mathematical Sciences]}.
\newblock Springer, Heidelberg, 2013.
\newblock The geometry of numerical algorithms.

\bibitem{partialbergmankernel}
Dan Coman and George Marinescu.
\newblock On the first order asymptotics of partial {B}ergman kernels.
\newblock {\em Ann. Fac. Sci. Toulouse Math. (6)}, 26(5):1193--1210, 2017.

\bibitem{daima}
Xianzhe Dai, Kefeng Liu, and Xiaonan Ma.
\newblock On the asymptotic expansion of {B}ergman kernel.
\newblock {\em J. Differential Geom.}, 72(1):1--41, 2006.

\bibitem{diatta}
Daouda~Niang Diatta and Antonio Lerario.
\newblock Low degree approximation of random polynomials, 2018,
  arXiv:1812.10137.

\bibitem{gwexp}
Damien Gayet and Jean-Yves Welschinger.
\newblock Exponential rarefaction of real curves with many components.
\newblock {\em Publ. Math. Inst. Hautes \'Etudes Sci.}, (113):69--96, 2011.

\bibitem{gwlowerbound}
Damien Gayet and Jean-Yves Welschinger.
\newblock Lower estimates for the expected {B}etti numbers of random real
  hypersurfaces.
\newblock {\em J. Lond. Math. Soc. (2)}, 90(1):105--120, 2014.

\bibitem{gw1}
Damien Gayet and Jean-Yves Welschinger.
\newblock What is the total {B}etti number of a random real hypersurface?
\newblock {\em J. Reine Angew. Math.}, 689:137--168, 2014.

\bibitem{gw2}
Damien Gayet and Jean-Yves Welschinger.
\newblock Betti numbers of random real hypersurfaces and determinants of random
  symmetric matrices.
\newblock {\em J. Eur. Math. Soc. (JEMS)}, 18(4):733--772, 2016.

\bibitem{harnack}
Axel Harnack.
\newblock Ueber die {V}ieltheiligkeit der ebenen algebraischen {C}urven.
\newblock {\em Math. Ann.}, 10(2):189--198, 1876.

\bibitem{hezarilu}
Hamid Hezari and Hang Xu.
\newblock Quantitative upper bounds for {B}ergman kernels associated to smooth
  {K}\"ahler potentials, 2018, arxiv:1807.00204.

\bibitem{itenbergDelPezzo}
Ilia Itenberg, Viatcheslav Kharlamov, and Eugenii Shustin.
\newblock Welschinger invariants of real del {P}ezzo surfaces of degree {$\geq
  2$}.
\newblock {\em Internat. J. Math.}, 26(8):1550060, 63, 2015.

\bibitem{itenbergviro}
Ilia Itenberg and Oleg Viro.
\newblock Asymptotically maximal real algebraic hypersurfaces of projective
  space.
\newblock In {\em Proceedings of {G}\"{o}kova {G}eometry-{T}opology
  {C}onference 2006}, pages 91--105. G\"{o}kova Geometry/Topology Conference
  (GGT), G\"{o}kova, 2007.

\bibitem{klein}
Felix Klein.
\newblock Ueber den {V}erlauf der {A}bel'schen {I}ntegrale bei den {C}urven
  vierten {G}rades.
\newblock {\em Math. Ann.}, 10(3):365--397, 1876.

\bibitem{kostscalarproduct}
E.~Kostlan.
\newblock On the distribution of roots of random polynomials.
\newblock In {\em From {T}opology to {C}omputation: {P}roceedings of the
  {S}malefest ({B}erkeley, {CA}, 1990)}, pages 419--431. Springer, New York,
  1993.

\bibitem{ma2}
Xiaonan Ma and George Marinescu.
\newblock {\em Holomorphic {M}orse inequalities and {B}ergman kernels}, volume
  254 of {\em Progress in Mathematics}.
\newblock Birkh\"auser Verlag, Basel, 2007.

\bibitem{nazarovsodinasymptotic}
F.~Nazarov and M.~Sodin.
\newblock Asymptotic laws for the spatial distribution and the number of
  connected components of zero sets of {G}aussian random functions.
\newblock {\em Zh. Mat. Fiz. Anal. Geom.}, 12(3):205--278, 2016.

\bibitem{Raffalli}
Christophe Raffalli.
\newblock Distance to the discriminant, 2014, arXiv:1404.7253.

\bibitem{logbergman}
Jingzhou Sun.
\newblock Logarithmic bergman kernel and conditional expectation of gaussian
  holomorphic fields, 2019, https://arxiv.org/abs/1912.10376.

\bibitem{thomreelle}
Ren\'{e} Thom.
\newblock Sur l'homologie des vari\'{e}t\'{e}s alg\'{e}briques r\'{e}elles.
\newblock In {\em Differential and {C}ombinatorial {T}opology ({A} {S}ymposium
  in {H}onor of {M}arston {M}orse)}, pages 255--265. Princeton Univ. Press,
  Princeton, N.J., 1965.

\bibitem{tian}
Gang Tian.
\newblock On a set of polarized {K}\"ahler metrics on algebraic manifolds.
\newblock {\em J. Differential Geom.}, 32(1):99--130, 1990.

\bibitem{welthese}
Jean-Yves Welschinger.
\newblock {\em Courbes alg\'{e}briques r\'{e}elles et courbes flexibles sur les
  surfaces r\'{e}gl\'{e}es}, volume 2000/43 of {\em Pr\'{e}publication de
  l'Institut de Recherche Math\'{e}matique Avanc\'{e}e [Prepublication of the
  Institute of Advanced Mathematical Research]}.
\newblock Universit\'{e} Louis Pasteur, D\'{e}partement de Math\'{e}matique,
  Institut de Recherche Math\'{e}matique Avanc\'{e}e, Strasbourg, 2000.
\newblock Th\'{e}se, Universit\'{e} Louis Pasteur, Strasbourg, 2000.

\bibitem{wilson}
George Wilson.
\newblock Hilbert's sixteenth problem.
\newblock {\em Topology}, 17(1):53--73, 1978.

\bibitem{zel}
Steve Zelditch.
\newblock Szeg\"o kernels and a theorem of {T}ian.
\newblock {\em Internat. Math. Res. Notices}, (6):317--331, 1998.

\end{thebibliography}
\end{document}